\documentclass[10pt]{article}

\usepackage[english]{babel}
\usepackage[utf8x]{inputenc}
\usepackage{amsmath}
\usepackage{graphicx}
\usepackage[colorinlistoftodos]{todonotes}
\usepackage{amssymb}
\usepackage{amsfonts}
\usepackage{amsthm}
\usepackage{mathrsfs}

\newtheorem{thm}{Theorem}

\newtheorem{lem}[thm]{Lemma}
\newtheorem{clm}[thm]{Claim}
\newtheorem{prp}[thm]{Proposition}
\newtheorem{cor}[thm]{Corollary}
\newtheorem{fct}[thm]{Fact}

\def\f{{\phi}}
\def\k{{\kappa}}

\def\s{{\sigma}}

\def\oj{{\overline{j}}}

\def\cF{{\cal F}}
\def\cG{{\cal G}}
\def\cH{{\cal H}}
\def\cK{{\cal K}}

\def\sg12{{\s_{1,2}}}
\def\sig21{{\s_{2,1}}}

\def\2N{{2^{[n]}}}
\def\mt{{\emptyset}}

\def\n1cr1{{\binom{n-1}{r-1}}}
\def\Ncr{{\binom{[n]}{r}}}

\def\phinv{{\phi^{-1}}}
\def\pr{{\prime}}
\def\rar{{\rightarrow}}
\def\sse{{\subseteq}}
\def\Sij{{\s_{i,j}}}
\def\Sji{{\s_{j,i}}}
\def\Spij{{\s^\pr_{i,j}}}

\author{
	Glenn Hurlbert\thanks{
		Department of Mathematics and Applied Mathematics,
		Virginia Commonwealth University
	}
	\thanks{
		\texttt{ghurlbert@vcu.edu}.
		Research partially supported by Simons Foundation Grant \#246436.
	}\\
Vikram Kamat\thanks{
	Department of Mathematics and Computer Science,
	University of Richmond
}
\thanks{
	\texttt{vkamat@richmond.edu}.
}\\
}

\title{New injective proofs of the Erd\H{o}s--Ko--Rado and Hilton--Milner theorems}

\begin{document}
	\maketitle
	
	\begin{abstract}
		A set system $\cF$ is \textit{intersecting} if for any $F, F'\in \cF$, $F\cap F'\neq \emptyset$. 
		A fundamental theorem of Erd\H{o}s, Ko and Rado states that if $\cF$ is an intersecting family of $r$-subsets of $[n]=\{1,\ldots,n\}$, and $n\geq 2r$, then $|\cF|\leq \n1cr1$. 
		Furthermore, when $n>2r$, equality holds if and only if $\cF$ is the family of all $r$-subsets of $[n]$ containing a fixed element. This was proved as part of a stronger result by Hilton and Milner. 
		In this note, we provide new injective proofs of the Erd\H{o}s--Ko--Rado and the Hilton--Milner theorems.
	\end{abstract}

	\section{The Erd\H{o}s--Ko--Rado theorem}
	
	For $0\leq j\leq n$, let $[j,n]=\{j,\ldots,n\}$. 
	In particular, set $[n]=[1,n]$. 
	Similarly, define $(j,n)=\{j+1,\ldots,n-1\}$. 
	For a set $X$ and $1\leq r\leq |X|$, denote $2^X=\{A:A\sse X\}$ and $\binom{X}{r}=\{A\in 2^X:|A|=r\}$. 
	A family $\cF\sse \Ncr$ is called $r$-{\it uniform}, with  $\cF_x=\{F\in\cF: x\in F\}$ called its {\it star centered at} $x$. 
	A {\it full star} is $\Ncr_x$ for some $x$; it is easy to see that $|\Ncr_x|=\n1cr1$. 
	We say that $\cF$ is {\it intersecting} if $A\cap B\neq\mt$ for every $A,B\in\cF$.
	\\
	\\
	One of the central results in extremal set theory, the Erd\H{o}s--Ko--Rado theorem finds a tight upper bound on the size of uniform intersecting set systems. As part of a stronger result that characterized the size and structure of the ``second best'' intersecting set systems, Hilton and Milner \cite{HilMil} proved that the extremal structures are essentially (up to isomorphism) unique.
	\\
	\begin{thm}
		\cite{EKR,HilMil}
		\label{EKR}
		If $1\le r\le n/2$ and $\cF\sse\Ncr$ is intersecting, then $|\cF|\le\n1cr1$.
		If $r<n/2$ then equality holds if and only if $\cF=\Ncr_x$ for some $x\in [n]$.
	\end{thm}
	\ \\
	A cornerstone of extremal combinatorics, the theorem has inspired a multitude of research avenues and applications (see \cite{DezFra,FranklShift, FraTok, GodMea,HurKamChor}). 
	The original proof by Erd\H{o}s, Ko and Rado made use of the now-central \textit{shifting} technique in conjunction with an induction argument. 
	Daykin \cite{Daykin} later discovered that the theorem is implied by the Kruskal-Katona theorem \cite{KK1, KK2}, while Katona \cite{Katona} gave possibly the simplest proof using the notion of \textit{cyclic permutations}. Most recently, Frankl and F\"uredi \cite{FranklFuredi} provided another new short proof of the theorem using a non-trivial result of Katona \cite{KatShadow} on shadows of intersecting families.
	\\
	\\
	The new proof we provide is closest in spirit to the original proof, but avoids induction and counting, and is as short as any.
	It relies on the shifting operation and some of its structural properties to construct an injective function that maps any intersecting family to a subfamily of $\Ncr_1$.
	While the shifting operation is injective, it is not explicitly so; that is, the shift operation on a set depends on the entire family.
	However, our new injection for shifted families is explicit. 
	By direct comparison, while the approach of \cite{FranklFuredi} uses an explicit complementation followed by a shadow bound, our approach uses shifting followed by an explicit complementation. Finally, as mentioned earlier, our technique also helps recover a new short proof of the Hilton--Milner theorem (Theorem \ref{HM}), which we describe in the final section.  We also note here that Borg \cite{BorgTwo} used an injective argument to prove an analog of the Erd\H{o}s--Ko--Rado theorem for integer partitions.

	\section{Shifting}
	
	We begin by reviewing the definition of the renowned shifting operation and state some of its important properties. 
	For set $A\sse [n]$ and $x\in [n]$, let $A+x=A\cup\{x\}$, $A-x=A\setminus\{x\}$.
	\\
	\\
	Define the $(i,j)$-{\it shift} $\Sij:\2N\rar\2N$ as follows: for $A\in \2N$, let $\Sij(A)=A-i+j$ if $i\in A$ and $j\not\in A$, and $\Sij(A)=A$ otherwise.
	Extend this definition to $\Sij:2^\2N\rar 2^\2N$ as follows: for $\cF \sse \2N$, let $\Sij(\cF)=\{\Spij(A): A\in\cF\}$, where $\Spij(A)=\Sij(A)$ if $\Sij(A)\not\in\cF$, and $\Spij(A)=A$ otherwise. 
	The following facts are well known and easy to verify.
	\\
	\begin{fct}
		\label{sizes}
		For all $A\sse [n]$ and all $\cF\sse\2N$ we have
		\begin{enumerate}
			\item
			$|\Sij(A)|=|A|$,
			\item
			$|\Sij(\cF)|=|\cF|$, and
			\item
			If $\cF$ is intersecting then so is $\Sij(\cF)$.
		\end{enumerate}
	\end{fct}
	\ \\
	We say that a family $\cF\sse \Ncr$ is \textit{shifted} if for any $1\leq j<i\leq n$, $\Sij(\cF)=\cF$. 
	Frankl \cite{FranklShift} proved the following useful proposition about shifted families.
	\\
	\begin{prp}\label{FranklProp}
		Let $\mathcal{F}\sse \binom{[n]}{r}$ be shifted and intersecting. 
		Then for every $F\in \cF$, there exists a $k=k(F)$ such that $|F\cap [2k+1]|\geq k+1$.
	\end{prp}
	\ \\
	The following corollary of Proposition \ref{FranklProp} is immediate, and will be used in the proof of Claim \ref{inj}.
	\\
	\begin{cor}\label{NewCor}
		Let $\cF\sse \binom{[n]}{r}$ be shifted and intersecting, and let $r\leq n/2$. 
		Then for every $F\in \cF$, there exists a $k=k(F)$ such that $|F\cap [2k]|=k$.
	\end{cor}
	
	\begin{proof}
		Let $F\in \cF$ and let $k=k(F)$ be maximum such that $|F\cap [2k]|\geq k$. 
		From Proposition \ref{FranklProp}, we know that such a $k$ exists. 
		We claim that $|F\cap [2k]|= k$. 
		If $2k=n$, then we have $|F\cap [n]|=r\leq \frac{1}{2}(2k)=k$, which implies the result, so we assume that $2k<n$. 
		Suppose that $|F\cap [2k]|\geq k+1$. 
		First, this implies that $n\geq 2k+2$. 
		Next, the maximality of $k$ implies that $|F\cap [2k+2]|\leq k$, a contradiction. 
		Thus $|F\cap [2k]|=k$.
	\end{proof}

	\section{Proof of Theorem \ref{EKR}}
	
 For intersecting $\cF\sse\binom{[n]}{r}$ with $1\le r\le n/2$,
	we shift $\cF$ until it becomes the shifted, intersecting family $\cF'$. Now define the function $\f:\cF'\to \Ncr_1$ as follows.
	For a set $F\in \cF'$, let $\k=\k_F$ be maximum such that $|F\cap [2\k]|=\k$. 
	We know that $\k$ exists, from Corollary \ref{NewCor}. 
	Now, if $1\in F$, let $\f(F)=F$; otherwise, let $\f(F)=F\triangle [2\k]$.
	We also denote $\phi(\cF)=\{\phi(A): A\in\cF\}$, as well as write $\phinv(B)=A$ whenever $\phi(A)=B$, with $\phinv(\cH)=\{\phinv(B): B\in\cH\}$.
	\\
    \\
    \noindent
	Fact \ref{sizes} gives $|\cF|=|\cF'|$, and Claim \ref{inj} below gives $|\cF'|\le\n1cr1$.
	When $r<n/2$, Lemma \ref{invphi} shows that $\cF'$ is a full star, and Lemma \ref{preshift} below shows that $\cF$ is a full star.\hfill $\Box$
    \\
    \\
    \noindent
    We now prove Claim \ref{inj} and Lemmas \ref{preshift} and \ref{invphi} in the subsections below.

	\subsection{Injection}

	\begin{clm}\label{inj}
		For $r\leq n/2$, if $\mathcal{F}\sse\Ncr$ is shifted and intersecting then the function $\f$ defined above is injective.
	\end{clm}
	
	\begin{proof}
		Let $F_1,F_2\in \cF$, $F_1\neq F_2$. 
		If $1\in F_1$ and $1\in F_2$ then it is obvious that $\f(F_1)\neq \f(F_2)$. 
		\\
		\\
		Suppose that $1\notin F_1$ and $1\notin F_2$. 
		If $\k=\k_{F_1}=\k_{F_2}$ then $F_1\cap [2\k]\neq F_2\cap [2\k]$ or $F_1\setminus [2\k]\neq F_2\setminus [2\k]$.
		Then the definition of $\f$ implies that $\f(F_1)\neq \f(F_2)$, as required. 
		So, without loss of generality, we may assume that $\k_{F_1}<\k_{F_2}$. 
		Using maximality of $\k_{F_1}$, we have that $F_1\setminus [2\k_{F_2}]\neq F_2\setminus [2\k_{F_2}]$. 
		As $F_1\setminus [2\k_{F_2}]\sse \f(F_1)$ and $F_2\setminus [2\k_{F_2}]\sse \f(F_2)$, this implies that $\f(F_1)\neq \f(F_2)$.
		\\
		\\
		Finally, suppose $1\in F_1$ and $1\notin F_2$. 
		We need to show that $\f(F_2)\neq F_1$. 
		Suppose instead that $\f(F_2)=F_1$. 
		Let $G_2=F_2\setminus [2\k_{F_2}]$ and break $G_2$ into its maximum intervals. 
		That is, we write $G_2=\cup_{i=0}^{p-1}[t_i,s_{i+1}]$, where 
        $2\k_{F_2}=s_0<t_0$, 
        $t_i\le s_{i+1}$ for each $0\le i<p$,
        $s_i+1<t_i$ for each $0<i<p$, and
        $t_p=n$. 
		For every $h\in [0,p-1]$, we can see that $|\cup_{i=0}^h (s_{i},t_{i})|>|\cup_{i=0}^h [t_{i},s_{i+1}]|$, which we refer to as {\it Property $\star$}.
        Indeed, for each $1\le j\le n$ define $X_j=|F_2\setminus [j]|-|F_2\cap [j]|$.
        Then $X_{2\k_{F_2}}=0$, $X_{n}>0$, and $|X_j-X_{j+1}|=1$ for all $j$.
        If ever we have $|\cup_{i=0}^h (s_{i},t_{i})|\leq |\cup_{i=0}^h [t_{i},s_{i+1}]|$ for some $h\in [0,p-1]$ then $X_{s_{h+1}}\le 0$.
        Thus there is some $l\in [s_{h+1},n]$ such that $X_l=0$ (the discrete mean value theorem); i.e. $|F_2\cap [l]|=l/2$.
        This contradicts the maximality of $\k_{F_2}$. 
                
		Let $P\sse \cup_{i=0}^{p-1} (s_{i},t_{i})$ be the set of the smallest $|G_2|$ elements in $\cup_{i=0}^{p-1} (s_{i},t_{i})$. 
		It is easy to see that $P\cap G_2=\emptyset$. 
		Also, because of Property $\star$, $P$ can be obtained from $G_2$ by a sequence of $(i,j)$-shifts $\Sij$. 
		Consequently, as $\cF$ is shifted, $F_2^\pr=F_2\setminus G_2\cup P\in \cF$. 
		However, from the definition of $\f$, and under the assumption that $\f(F_2)=F_1$, we have $G_2\sse F_1$. 
		This implies that $F_1\cap F_2^\pr=\emptyset$, a contradiction, as $\mathcal{F}$ is intersecting. 
	\end{proof}
    
    \noindent
	We make note of the following interesting property of the parameter $\k$.
	If $\cF\sse\Ncr$ is shifted and intersecting, with $A\in\cF$ and $B=\phi(A)$, then $|B\cap [2k]|=|A\cap [2k]|$ for all $k\ge \k_A$.
	This makes it possible to define $\k_B$ similarly, from which we see that $\k_B=\k_A$.
	Consequently, if we know that $B\in\phi(\cF)$ then it must be that $\phinv(B)\in\{B,B\triangle [2\k_B]\}$.
	We phrase this as follows.
	
	\begin{prp}\label{phiprop}
		For $r\leq n/2$, if $\cF\sse\Ncr$ is shifted and intersecting, with $B\in\phi(\cF)\setminus\cF$, then $\phinv(B)=B\triangle [2\k_B]$.\hfill $\Box$
	\end{prp}

	\subsection{Star Preservation}
	Here we show that the pre-shift of any full star is a full star, and also that $\phinv(\binom{[n]}{r}_1)=\binom{[n]}{r}_1$.
	Let $G(M,s)$ be the graph on the vertex set $\binom{M}{s}$ having edge $AB$ (for any $A, B\in \binom{M}{s}$) whenever $|A\triangle B|=2$.
	
	\begin{fct}\label{conn}
		For $0\le s\le |M|$ we have that $G(M,s)$ is connected.
	\end{fct}
	
	\begin{proof}
		The standard revolving door algorithm (Gray code for uniform subsets; see Algorithm R in Section 7.2.1.3 of \cite{Knuth}) shows that $G(M,s)$ is hamiltonian.
	\end{proof}
	
	\begin{prp}\label{tradelemma}
		For $1\le r<n/2$ and intersecting $\cF\subset\binom{[n]}{r}$ with $|\cF|=\n1cr1$, let $\cH=\Sji(\cF)$ for some $i,j\in [n]$.
		Suppose that $\cH=\binom{[n]}{r}_i$ and define $M=[n]\setminus\{i,j\}$.
		\begin{enumerate}
			\item \label{disjointprop}
			If $A, C\in\binom{M}{r-1}$, $A\cap C=\emptyset$, and $A+j\in\cF$ then $C+j\in\cF$.
			\item \label{tradeprop}
			If $A, B\in\binom{M}{r-1}$, $AB\in E(G(M,r-1))$, and $A+j\in\cF$ then $B+j\in\cF$.
		\end{enumerate}
	\end{prp}
	
	\begin{proof}
		For part (\ref{disjointprop}), suppose instead that $C+j\not\in\cF$.
		Then $C+i\in\cF$ because $C+i\in\cH$.
		But then $(A+j)\cap (C+i)=\emptyset$, a contradiction.
		\\
		\\
		For part (\ref{tradeprop}), given such $A$ and $B$, let $C\in\binom{M\setminus (A\cup B)}{r-1}$, which is possible because $|M\setminus (A\cup B)|=(n-2)-r\ge r-1$.
		Then two applications of part (\ref{disjointprop}) yields the result.
	\end{proof}
	
	\begin{lem}\label{preshift}
		For $1\le r<n/2$ and intersecting $\cF\subset\binom{[n]}{r}$ with $|\cF|=\n1cr1$, let $\cH=\Sji(\cF)$ for some $i,j\in [n]$.
		Suppose that $\cH=\binom{[n]}{r}_i$.
		Then either $\cF=\binom{[n]}{r}_i$ or $\cF=\binom{[n]}{r}_j$.
	\end{lem}
\noindent 
We note that Borg \cite{BorgOne} proved a more general form of this lemma; however, for the sake of completeness and the reader's convenience, we provide a short proof below. 
	\begin{proof}
		Suppose that $\cF\not=\binom{[n]}{r}_i$, and define $M=[n]\setminus\{i,j\}$.
		Since $\Sji(\cF)=\cH$, every set in $\cF$ must contain either $i$ or $j$.
		Thus there must be some $A\sse M$ such that $A+j\in\cF$.
		By Fact \ref{conn} and Proposition \ref{tradelemma}, $\{A+j:A\in \binom{M}{r-1}\}\sse\cF$.
		\\
		\\
		Also, if $j\in S\in\cH$ then $S\in\cF$.
		Hence $\cF=\binom{[n]}{r}_j$.
	\end{proof}
	
	\begin{lem}\label{invphi}
		For $1\le r<n/2$ and shifted, intersecting $\cF\subset\binom{[n]}{r}$ with $\phi(\cF)=\binom{[n]}{r}_1$, we have $\cF=\binom{[n]}{r}_1$.
	\end{lem}
	
	\begin{proof}
		Suppose first that $A=\{1,n-r+2,\ldots,n\}\in \cF$. 
		As $\cF$ is shifted, this implies that $\binom{[n]}{r}_1\sse\cF$, as required. 
		Thus, we may assume that $A\notin \cF$. 
		Since $A\in \phi(\cF)$, we have by Proposition \ref{phiprop} that $\phi^{-1}(A)=\{2,n-r+2,\ldots,n\}\in \cF$. 
		However, because $\cF$ is shifted, we obtain that $A\in \cF$, a contradiction. 
	\end{proof}

	\section{The Hilton--Milner theorem}
	Hilton and Milner \cite{HilMil} characterized the structure of maximum non-star intersecting families. More precisely, they proved the following statement. For $1\leq r\leq n-1$, let $C=[2,r+1]$. Let $\mathcal{H}=\{F\in \Ncr:1\in F,F\cap C\neq \emptyset\}\cup \{C\}$ and $\cK=\{K\in\binom{n}{3}: |K\cap [1,3]|\ge 2\}$.
    \begin{thm}[Hilton--Milner]\label{HM}
    For $2\le r<n/2$, let $\cF\subseteq \Ncr$ be intersecting such that $\bigcap_{F\in \cF}F=\emptyset$. Then $|\cF|\leq \binom{n-1}{r-1}-\binom{n-r-1}{r-1}+1$ and equality holds if and only if $\cF\cong\mathcal{H}$ or $r=3$ and $\cF\cong\cK$.
    \end{thm}

	\begin{proof}
Using a shifting idea similar to the one used by Frankl and F\"uredi \cite{FrankFur} in their inductive proof of the Hilton--Milner theorem, we construct an injection that maps any non-star intersecting family to a subfamily of $\mathcal{H}$ as follows.
   Let $\cF\subseteq \Ncr$ be an intersecting non-star with $r<n/2$ and $1\not\in F\in\cF$.
   \\
   \\
Perform shifts on $\cF$ until either it becomes a star or is shifted.
The latter case results in the non-star, shifted $\cF'$.
Since it is non-star, some $F\in\cF'$ does not contain $1$. 
Because it is shifted, $C\in\cF'$. 
Note that $\binom{[r+1]}{r}\sse\cF'$.
   \\
   \\
The former case leads to intermediate families $\cF^1$ (non-star) and $\cF^2$ such that $\sigma_{x,y}(\cF^1)=\cF^2$ and $y\in \bigcap_{F\in \cF^2}F$. 
   Clearly, for each $F\in \cF^1$, $F\cap \{x,y\}\neq \emptyset$. 
   Without loss of generality (relabeling if necessary), assume $x=2$ and $y=1$. 
   To $\cF^1$, we apply all shifting operations $\sigma_{i,j}$ with $2\leq j<i\leq n$ to obtain $\cF'_1$. 
   Note that $C\in \cF'_1$. Also, as $\cF_1$ is non-star, there exists some $G\in \cF_1$ such that $G\cap \{1,2\}=\{1\}$. As $\cF'_1$ is shifted, this implies that $\{1,3,\ldots,r+1\}\in \cF'_1$.
   \\
   \\
   Let $\mathcal{F}^*=\cF'_1\cup \{G\in \Ncr:\{1,2\}\subseteq G\}$. 
   This implies that $\binom{[r+1]}{r}\subseteq \mathcal{F}^*$. 
   Finally, apply all shifts to $\mathcal{F}^*$ until we obtain a shifted intersecting family $\mathcal{F}'$. 
   The fact that $\binom{[r+1]}{r}$ is unchanged by any shifting operation means $\cF'$ is also a non-star family. 
   \\
   \\
  Now, define the injection $\phi':\cF'\to \Ncr_1$ as $\phi'(C)=C$ and $\phi'(F)=\phi(F)$ otherwise. We only have to show that for each $F\in \cF'\setminus \{C\}$, $\phi'(F)\cap C\neq \emptyset$. If $1\in F$, then this is obvious as $\cF'$ is intersecting and $\phi'(F)=F$, so suppose $1\notin F$. Let $\kappa=\kappa_F\leq r$. Clearly, there exists an $x\in (C\cap [2\kappa]) \setminus (F\cap [2\kappa])$ (otherwise $F=C$), which implies $x\in \phi'(F)$ as required.
  \\
  \\
  The characterization of extremal families is carried out in Lemmas \ref{invphiH}, \ref{preshiftK}, and \ref{preshiftH}, below.
	\end{proof}
  
	\begin{lem}\label{invphiH}
		For $2\le r<n/2$ and shifted, intersecting $\cF\subset\binom{[n]}{r}$ with $\phi'(\cF)=\cH$, we have $\cF=\cH$ or $r=3$ and $\cF=\cK$.
	\end{lem}
	
	\begin{proof}
    The case $r=2$ is trivial, so assume that $r\ge 3$.
		Suppose first that $A=\{1,r+1,n-r+3,\ldots,n\}\in \cF$. 
		As $\cF$ is shifted, this implies that $\cH-\{C\}\sse\cF$, as required. 
		Thus, we may assume that $A\notin \cF$ (so $\cF\neq\cH$). 
		Since $A\in \phi'(\cF)$, we have by Proposition \ref{phiprop} that either $\phi'^{-1}(A)=\{2,3,n\}\in \cF$ when $r=3$ (since $n\ge 7$) or $\phi'^{-1}(A)=\{2,r+1,n-r+3,\ldots,n\}\in \cF$ when $r\ge 4$.
		Because $\cF$ is shifted, we obtain either that $\phi'^{-1}(\cH)=\cK$ when $r=3$ or that $A\in \cF$ when $r\ge 4$, a contradiction. 
	\end{proof}
\noindent
Backing up further, the maximality of $|\cF|$ implies that $\{G\in\Ncr: \{1,2\}\sse G\}\sse\cF'_1$, so that $\cF^*=\cF'_1$.
\\
\\
It is fairly easy to see that if $\cF\cong\cK$ then any $\sigma_{i,j}(\cF)\cong\cK$, and so $\cF'=\cK$, and then $\phi'(\cK)=\cK$.
Similarly, if $\cF\cong\cH$ then any $\sigma_{i,j}(\cF)\cong\cH$, and so $\cF'=\cH$, and then $\phi'(\cK)=\cH$.
The proofs rely on the idea of symmetry: there are two types of elements in $\cK$ and three types in $\cH$; the shift $\sigma_{i,j}$ does not change either family when $i$ and $j$ have the same type, and swaps the types when their types differ.
The converse of these statements is recorded in the following two lemmas.

\begin{lem}\label{preshiftK}
If $\cG\cong\cK$ then $\sigma_{i,j}^{-1}(\cG)\cong\cK$.
\end{lem}

\begin{proof}
Define $X$ and $Y=[n]\setminus X$ so that $\cG=\{X\}\cup \{G\in\binom{[n]}{3}: |G\cap X|=2\}$, and denote $\cG'=\sigma_{i,j}^{-1}(\cG)$.
\begin{enumerate}
\item
Case: $i,j\in X$.
\\
\\
It must be that $\cG'=\cG$; otherwise there must be some $k\in X\setminus\{i,j\}$ and $y\in Y$ such that $\sigma_{i,j}^{-1}(\{j,k,y\})=\{i,k,y\}$.
But this would mean that $\{i,k,y\}\not\in\cG$, a contradiction.
\\
\item
Case: $i,j\in Y$.
\\
\\
It must be that $\cG'=\cG$; otherwise there must be some $\{k,l\}\subset X$ such that $\sigma_{i,j}^{-1}(\{k,l,j\})=\{k,l,i\}$.
But this would mean that $\{k,l,i\}\not\in\cG$, a contradiction.
\\
\\
\item
Case: $i\in X, j\in Y$. 
\\
\\
It must be that $\cG=\cG'$.
Indeed, for every set $Z\in\cG$ with $i\in Z$ or $j\not\in Z$ we must have $Z\in\cG'$.
Thus we only need to consider sets $Z\in\cG$ for which $i\not\in Z$ and $j\in Z$.
Hence $Z=X\setminus \{i\}\cup \{j\}$.
But as $X\in\cG\cap\cG'$ we must have $Z\in\cG'$.
\item
Case: $i\in Y, j\in X$. 
\\
\\
Suppose that $\cG'\neq\cG$.
Then there is some set $Z\in\cG'\setminus\cG$, which means that $i\in Z$ and $j\not\in Z$.
Let $X=\{j,k,l\}$ and $X'=\{i,k,l\}$.
Because $\{X,X'\}\subset\cG$, we know that $Z\neq X'$; without loss of generality, since $Z\setminus\{i\}\cup\{j\}\in\cG$, $Z=\{i,k,y\}$ for some $y\in Y$.
\\
\\
Now consider any set $W=\{j,l,w\}$ for $w\in Y\setminus\{y\}$.
Because $W\in\cG$ and $W\cap Z=\emptyset$, we must have that $W\not\in\cG'$ and consequently that $W'=\{i,l,w\}\in\cG'$.
\\
\\
Similarly, consider any set $V=\{j,k,v\}$ for $v\in Y\setminus\{w\}$.
Because $V\in\cG$ and $V\cap W=\emptyset$, we must have that $V\not\in\cG'$ and consequently that $V'=\{i,k,v\}\in\cG'$.
\\
\\
Finally, for every set $U=\{k,l,u\}$ for any $u\in Y$, we have that $U\in\cG\cap\cG'$.
Hence $\cG=\{X'\}\cup \{G\in\binom{[n]}{3}: |G\cap X'|=2\}$, and thus $\cG'\cong\cK$.

\end{enumerate}
\end{proof}

\begin{lem}\label{preshiftH}
If $\cG\cong\cH$ then $\sigma_{i,j}^{-1}(\cG)\cong\cH$.
\end{lem}
	
\begin{proof}
Define the partition $\{X,Y,\{z\}\}$ of $[n]$ so that $\cG=\{X\}\cup \{G\in\binom{[n]}{r}_z: G\cap X\neq\emptyset\}$, and denote $\cG'=\sigma_{i,j}^{-1}(\cG)$ and $X'=X\setminus \{j\}\cup \{i\}$.
We first note that $\cG_i\subset\cG'$. 
Also $\cG_{\oj}\subset\cG'$, where $\cG_{\oj}=\{G\in\cG: j\not\in\cG\}$.
If it is the case that $\cG'\neq\cG$, then there is some $Z'\in\cG'\setminus\cG$, which means that $i\in Z'$ and $j\not\in Z'$. 
Moreover, $Z=Z'\setminus\{i\}\cup\{j\}\in\cG$.
\begin{enumerate}
\item
Case: $i=z$. 
\\
\\
It must be that $\cG'=\cG$.
Otherwise, since the only set in $\cG$ without $i$ is $X$, we would have that $Z'=X'$.
However, this would mean that $Z'\in\cG$, a contradiction.
\\
\item 
Case: $i,j\in X$ or $i,j\in Y$.
\\
\\
It must be that $\cG'=\cG$; otherwise there must be some $K\sse X\setminus\{i,j\}$ and $S\sse Y\setminus\{i,j\}$ such that $\sigma_{i,j}^{-1}(\{z,j\}\cup K\cup S)=\{z,i\}\cup K\cup S$.
But this would mean that $\{z,i\}\cup K\cup S\not\in\cG$, a contradiction.
\\
\item 
Case: $i\in X, j\in Y$.
\\
\\
Again we argue that $\cG'=\cG$. 
Otherwise, as $i\in X$, we have $Z\cap (X\setminus \{i\})\neq \emptyset$. 
However this implies that $Z'\cap X\neq \emptyset$. 
Since $z\in Z'$, this implies $Z'\in \cG$, a contradiction. 
\\
\item 
Case: $i\in X, j=z$.
\\
\\
Suppose that $\cG'\neq\cG$.
Choose any $v\in Z'\setminus\{i\}$ and $u\not\in Z'\cup\{j\}$ such that, if $Z'\cap X=\{i,v\}$ then $u\in X$, and define $U'=Z'\setminus\{v\}\cup\{u\}$ and $U=U'\setminus\{i\}\cup\{j\}$.
We show that $U'\in\cG'$.
\\
\\
Since $i\in Z'\neq X$ and $n>2r$, we can choose a set $W$ containing $z$ and $u$ that intersects both $X\setminus Z'$ and $Y\setminus Z'$ and is disjoint from $Z'$; clearly $W\in \cG$.
Let $W'=W\setminus\{j\}\cup\{i\}$.
Because $W\cap Z'=\emptyset$ it must be that $W'\in \cG'$. 
Notice that $W'\cap U=\emptyset$.
Thus $U\not\in\cG'$, and so $U'\in\cG'$.
\\
\\
Using this argument repeatedly, we see that, for every set $V'$ for which $i\in V$, $j\not\in V$, and $V\cap (X\setminus\{i\}\neq\emptyset$, we have $V'\in\cG'$.
This implies that $\cG'=\{X'\}\cup\{G\in\binom{[n]}{r}_i: G\cap X'\neq\emptyset\}\cong\cH$.
\\
\item
Case: $i\in Y, j\in X$. 
\\
\\
Suppose that $\cG'\neq\cG$.
We first make note that, for every nonempty $U\subsetneq X\setminus\{j\}$ and every $V\in\binom{Y\setminus\{i\}}{r-2-|U|}$, both $S\cup\{i\}\in\cG$ and $S\cup\{j\}\in\cG$, where $S=\{z\}\cup U\cup V$.
Hence both $S\cup\{i\}\in\cG'$ and $S\cup\{j\}\in\cG'$.
\\
\\
Next, for every nonempty $U\subsetneq X\setminus\{j\}$ and every $V\in\binom{Y\setminus\{i\}}{r-1-|U|}$, we have $S=\{z\}\cup U\cup V\in\cG$.
Hence $S\in\cG'$ as well.
Similarly, for every nonempty $U\subsetneq X\setminus\{j\}$ and every $V\in\binom{Y\setminus\{i\}}{r-3-|U|}$, we have $S=\{z,i,j\}\cup U\cup V\in\cG$, and thus $S\in\cG'$ also.
(Note that $S=\{z,i,j\}$ when $r=3$.)
\\
\\
Therefore we know that $Z'\cap X=\emptyset$, and so $X\not\in\cG'$, implying that $X'\in\cG'$.
For every remaining set $S\in\cG$ we have $S=\{z,j\}\cup V$ for some $V\subset Y$.
Since $S\cap X'=\emptyset$, we must have $S\setminus\{j\}\cup\{i\}\in\cG'$.
Thus $\cG'=\{X'\}\cup\{G\in\binom{[n]}{r}_z: G\cap X'\neq\emptyset\}\cong\cH$.
\\
\item
Case: $i\in Y, j=z$. 
\noindent
\begin{enumerate}
\item
Subcase: For every $x\in X$ we have $X\setminus\{x\}\cup\{j\}\in\cG'$.
\\
\\
This implies that $V\cup\{x,j\}\in\cG'$ for all $V\in\binom{Y\setminus\{i\}}{r-2}$.
\\
\\
Suppose that, for some $1\le k<r$, we have $\{j\}\cup U\cup V\in\cG'$ for all $U\in\binom{X}{k}$ and $V\in\binom{Y\setminus\{i\}}{r-k-1}$.
Then we claim that $\{j\}\cup U'\cup V'\in\cG'$ for all $U'\in\binom{X}{k+1}$ and $V'\in\binom{Y\setminus\{i\}}{r-k-2}$.
Indeed, choose such a $U'$ and $V'$, let $y\in Y\setminus\{i\}$ and $V''\ \sse\ Y\setminus (V'\cup\{i,y\})$ for some $|V''|=k$, and define $U''=X\setminus U'$.
Finally, choose $U\in\binom{U'}{k}$ and set $V=V'\cup\{y\}$.
Then $\{j\}\cup U\cup V\in\cG'$ and, since $(\{j\}\cup U\cup V)\cap (U''\cup V''\cup\{i\})=\emptyset$, we have $\{j\}\cup U''\cup V''\in\cG'$.
Similarly, because $(\{j\}\cup U''\cup V'')\cap (U'\cup V'\cup\{i\})=\emptyset$, we have $\{j\}\cup U'\cup V'\in\cG'$.
\\
\\
By induction, we have that $\cG'=\cG$.
\\
\item
Subcase: There is some $x\in X$ such that $Z=X\setminus\{x\}\cup\{j\}\not\in\cG'$ and $Z'=X\setminus\{x\}\cup\{i\}\in\cG'$.
\\
\\
This implies that $V\cup\{x,i\}\in\cG'$ for every $V\subset\binom{Y\setminus\{i\}}{r-2}$.
Next we claim that, for any $x'\in X\setminus\{x\}$ we have $V'\cup\{x',i\}\in\cG'$ for every $V'\subset\binom{Y\setminus\{i\}}{r-2}$. 
If not, then $V'\cup \{x',j\}\in \cG'$ and $(V'\cup \{x',j\})\cap Z'=\emptyset$, a contradiction.
\\
\\
Now suppose that for some $1\leq k<r$ we have $U\cup V\cup\{i\}\in\cG'$ for every $\emptyset\neq U\subsetneq X$, $|U|=k$ and $V\subset Y\setminus\{i\}$ with $|V|=r-k-1$.
We claim that we also have $U'\cup V'\cup\{i\}\in\cG'$ for every $\emptyset\neq U'\subsetneq X$, $|U'|=k+1$ and $V'\subset Y\setminus\{i\}$ with $|V'|=r-1-|U'|$. 
Suppose not. 
Then $U'\cup V'\cup \{j\}\in \cG'$. 
Choose sets $U$ and $V$ such that $U\subseteq X\setminus U'$, $V\subseteq Y\setminus (V'\cup \{i\})$ and $|U|\leq k$. 
As $n>2r$, such a choice of $U$ is always possible. 
Now, as $U\cup V\cup \{i\}\in \cG'$ by the induction hypothesis and $(U\cup V\cup \{i\})\cap (U\cup V\cup \{i\})=\emptyset$, this is a contradiction. 
\\
\\
By induction we have $\cG'=\{X\}\cup \{G\in\binom{[n]}{r}_i: G\cap X\neq\emptyset\}\cong\cH$.
\end{enumerate}
\end{enumerate}
\end{proof}

\end{document}